\documentclass[10pt]{article}
\usepackage{array}
\RequirePackage[OT1]{fontenc}
\RequirePackage{amsthm,amsmath,amsfonts}
\usepackage[authoryear]{natbib}
\usepackage{epsfig}
\usepackage{graphicx}

\numberwithin{equation}{section}
\theoremstyle{plain}
\newtheorem{thm}{Theorem}[section]
\newtheorem{proposition}{Proposition}[section]
\newtheorem{lem}{Lemma}[section]

\begin{document}

\renewcommand{\baselinestretch}{1.2}

\markboth{\hfill{\footnotesize\rm YING DING AND BIN NAN} \hfill}
{\hfill {\footnotesize\rm Mean Survival Time Estimation In Censored Linear Model} \hfill}

\renewcommand{\thefootnote}{}
$\ $\par


\fontsize{10.95}{14pt plus.8pt minus .6pt}\selectfont
\vspace{0.8pc}
\leftline{\large\bf ESTIMATING MEAN SURVIVAL TIME: WHEN IS IT POSSIBLE?}

\vspace{.4cm}
\centerline{YING DING}
\centerline{\it University of Pittsburgh}
\vspace{.4cm}
\centerline{BIN NAN}
\centerline{\it University of Michigan}
\vspace{.55cm}
\fontsize{9}{11.5pt plus.8pt minus .6pt}\selectfont


\noindent{\bf ABSTRACT. For right censored survival data, it is well known that the mean survival time can be consistently estimated when the support of the censoring time contains the support of the survival time. In practice, however, this condition can be easily violated because the follow-up of a study is usually within a finite window. In this article we show that the mean survival time is still estimable from a linear model when the support of some covariate(s) with nonzero coefficient(s) is unbounded regardless of the length of follow-up. This implies that the mean survival time can be well estimated when the covariate range is wide in practice. The theoretical finding is further verified for finite samples by simulation studies. Simulations also show that, when both models are correctly specified, the linear model yields reasonable mean square prediction errors and outperforms the Cox model, particularly with heavy censoring and short follow-up time.}

\vspace{9pt}


\fontsize{10.95}{14pt plus.8pt minus .6pt}\selectfont

\section{Introduction} \label{sec:intro1}

Estimating mean survival time becomes increasingly important, especially in oncology studies. For example, \citet{Zhao2011} proposed to use patient's mean survival time as the clinical outcome for the evaluation of optimal personalized treatment in lung cancer clinical trials. In some circumstances, the restricted mean survival time has been used to either compare the treatment effects or predict the individual patient's survival outcomes. However, it depends on the choice for the upper limit time $T^*$ and it is often difficult to explain the restriction part to clinicians or patients, whereas the unrestricted mean survival time provides an intuitive and straightforward interpretation.

The linear regression model for censored survival data, as an important alternative to the Cox model \citep{Cox1972}, has been extensively studied in recent years, see e.g. \citet{Ritov, Tsiatis, Wei1990, Ying1993, Jin2003, DingNan2011}, among many others. This type of model appeals in many ways \citep{Cox1984}: it models the failure time directly and thus has a more intuitive interpretation; it may provide more accurate summary of the data when Cox's proportional hazards assumption is violated; and more importantly, it can be used to predict the failure time in a straightforward way.

Obviously, a good mean survival time estimation/prediction requires a good estimator for the intercept parameter in a linear model. The study of such a linear model has primarily focused on the slope parameter estimation. Commonly used estimating methods for slopes include: the Buckley-James method \citep{BJ1979} that imputes the censored failure time using mean residual given covariates;  and the rank-based method \citep[among many others]{Prentice,Tsiatis,Ying1993} that is derived by using linear rank tests for the right censored data. \citet{Ritov} showed that the class of weighted rank-based estimating functions of \citet{Tsiatis} is asymptotically equivalent to the class of Buckley-James estimating functions on transformed residuals.

The estimation of intercept parameter when the error distribution is unspecified, however, has not been thoroughly studied mostly due to the finite follow-up time in practice so that the intercept is usually believed to be underestimated. \citet{BJ1979} claimed that in general the intercept can not be consistently estimated. In some of their simulations, however, \citet{Heller} and \citet{Schneider} found that the intercept can sometimes be estimated quite well. Without the presence of covariates, using an integration by parts argument with a truncation technique, \citet{Sursarla1980} showed that when the support of censoring time distribution contains the support of failure time distribution together with appropriate assumptions for the tail probability, the mean failure time estimation based on a Kaplan-Meier type estimator is consistent almost surely under random censoring. Using the reverse martingale approach, \citet{StuteWang} established more general strong consistency results including the mean failure time estimation.
Based on the work of \citet{Sursarla1984} and \citet{Sursarla1980}, \citet{WangNan} conjectured that the intercept can be consistently estimated when the supports of some covariates are not restricted to finite intervals. In this article, we confirm such a conjecture by formally establishing the consistency result for the intercept estimator. This result makes the estimation of mean survival time possible under a linear regression model when covariate support is wide in a practical setting.

In a linear model, the intercept estimation is equivalent to the mean failure time estimation on the residual scale if true values of the slope parameters are given. In reality, however, the slope parameters need to be estimated, which dramatically complicates the study of asymptotic properties of the intercept estimation. For the consistency of intercept estimation when slopes are estimated, we are only aware of \citet{LaiYing1991} who assumed bounded covariates, bounded support of the failure time distribution and wider support of the censoring time distribution. The latter assumption, however, is often violated in practice due to the nature of limited follow-up time in, for example, most of the human disease studies. Instead of assuming wider support of the censoring time distribution, we consider the setting that the supports of some covariates with nonzero coefficients are not restricted to finite intervals, which requires additional consideration on the slope estimation because its theoretical developments to date have been primarily under the assumption of bounded covariates. The unbounded covariate support is a technical condition, and approximates the practical situation where the ranges of the explanatory covariates are wide.

This article is organized as follows. In Section \ref{sec:IntEst} we present the strong consistency property of an intercept estimator under the assumption of unbounded covariates. In Section \ref{sec:SlpEst} we present both in probability and almost sure consistency as well as asymptotic normality results for the Gehan-weighted rank based slope estimators without assuming bounded covariates. In Section \ref{sec:Simu} we conduct simulation studies by varying the covariate support and the truncation time under different error distributions with different sample sizes. We also compare the failure time prediction performance with the Cox model under the standard extreme value error distribution for which both models fit the data correctly.
In Section \ref{sec:Exp} we provide an application to the Mayo primary biliary cirrhosis (PBC) study for illustration.
We provide some concluding remarks in Section \ref{sec:Con}. Proofs of the technical results rely on modern empirical process theory and are deferred to Section \ref{sec:Proofs}.

\section{Intercept estimation} \label{sec:IntEst}


Consider the linear regression model:
\begin{equation} \label{eq:aft}
T_i = \alpha_0 + X_i^{'}\beta_0 + \zeta_i, \quad i=1,\dots,n,
\end{equation}
where $\zeta_i$, $i=1,\dots,n$, are independent and identically distributed (i.i.d.) with zero mean. The response variable $T_i$ for the $i$th subject is the failure time (possibly transformed by a known monotone function). When $T_i$ is subject to right censoring, we only observe $(Y_i,\Delta_i,X_i)$, where $Y_i=\min(T_i,C_i)$, $C_i$ is the censoring time (possibly transformed by the same function that yields $T_i$), and $\Delta_i = 1(T_i\leq C_i)$. Here we assume that $(X_i, C_i)$, $i=1,\dots,n$, are i.i.d. and independent of $\zeta_i$.

Throughout the sequel we consider one-dimensional $\beta_0$ just for notational simplicity. All the results in this article hold for multiple-dimensional $\beta_0$.
Denote the parameter space for $\beta$ by ${\cal B}$, an arbitrary subset of the real line. For any $\beta \in \mathcal{B}$ we denote
\begin{equation*}
e_{\beta,i} = T_i - \beta X_i, \quad e_{0,i} = T_i-\beta_0 X_i=\alpha_0+\zeta_i,
\end{equation*}
and
\begin{equation*}
\epsilon_{\beta,i} = Y_i - \beta X_i, \quad \epsilon_{0,i}=Y_i-\beta_0 X_i.
\end{equation*}
Here, $e_{\beta,i}$ are the failure times in the residual scale with $\beta_0$ replaced by $\beta$, $\epsilon_{\beta,i}$ are the corresponding observed times in the residual scale for a fixed $\beta$, and $e_{0,i}$ are the error terms that have absorbed the intercept in model (\ref{eq:aft}). We use $F$ and $G$ to denote the distribution functions of $e_{0,i}$ and $C_i$, and $f$ and $g$ to denote their density functions, respectively. Now we adopt the empirical process notation of \citet{vanWellner1996}. In particular, for a function $f$ of a random variable $U$ that follows distribution $P$, we denote
\begin{eqnarray*}
P f = \int  f(u)\ d P(u), \quad
\mathbb{P}_nf = n^{-1}\sum_{i=1}^n f(U_i), \quad
\mathbb{G}_nf = n^{1/2}(\mathbb{P}_n-P)f,
\end{eqnarray*}
and refer all the details to the reference. Set $\epsilon_\beta=Y-\beta X$ and $\epsilon_0=Y-\beta_0 X$. Define
\begin{equation}
H_n^{(0)}(\beta,s) = \mathbb{P}_n \{1(\epsilon_\beta \leq s, \Delta = 1)\}, \
h^{(0)}(\beta,s) = P \{1(\epsilon_\beta \leq s, \Delta = 1)\}; \label{eq:proc1}
\end{equation}
and
\begin{equation}
H_n^{(1)}(\beta,s) = \mathbb{P}_n \{1(\epsilon_\beta \geq s)\}, \
h^{(1)}(\beta,s) = P \{1(\epsilon_\beta \geq s)\}. \label{eq:proc2}
\end{equation}

Since $\alpha_0 = Ee_{0,i} = \int_{-\infty}^{\infty} t\ dF(t)$, if the slope $\beta_0$ is known, then a natural estimator of $\alpha_0$ is given by
\begin{equation} \label{eq:alpha1}
\hat{\alpha}_n = \int_{-\infty}^{\infty} t\ d\hat{F}_n(t),
\end{equation}
where $\hat{F}_n(t)$ is the Kaplan-Meier estimator of the distribution function $F(t)$ of $e_0 = T-\beta_0 X$. In a regression setting, however, $\beta_0$ is unknown and hence needs to be estimated. Let $\hat{\beta}_n$ be an estimator of $\beta_0$, a direct extension of (\ref{eq:alpha1}) yields the estimator of interest:
\begin{equation} \label{eq:alpha2}
\hat{\alpha}_{n,\hat{\beta}_n} = \int_{-\infty}^{\infty} t\ d\hat{F}_{n,\hat{\beta}_n}(t),
\end{equation}
where $\hat{F}_{n,\beta}(t)$ is the Kaplan-Meier estimator of the distribution function $F_\beta(t)$ of $e_\beta=T-\beta X$, which is given by
\begin{equation} \label{eq:KM_b}
\hat{F}_{n,\beta}(t) = 1 - \prod_{i: \epsilon_{\beta,i}\leq t}\biggl\{1-\frac{\Delta_i/n}{H_n^{(1)}(\beta,\epsilon_{\beta,i})}\biggr\}.
\end{equation}
Note that the above estimator does not automatically provide a consistent estimator of $F_\beta(t)$ because $T-\beta X$ and $C-\beta X$ are not independent except when $\beta=\beta_0$. We will follow the method of \citet{LaiYing1991} to show that $\hat{F}_{n,\hat\beta_n}(t)$ converges to $F(t)$ when $\hat\beta_n$ converges to $\beta_0$ with certain polynomial rate.

\citet{Sursarla1984} showed that the above $\hat\alpha_{n, \hat\beta_n}$ is identical to the Buckley-James estimator of $\alpha_0$ for a fixed $\hat\beta_n$. When there is no covariates (equivalently  $\beta_0=0$) or $\beta_0$ is given, \citet{StuteWang} and \citet{Sursarla1980} studied the asymptotic properties of the mean survival time estimator (\ref{eq:alpha1}). They provided the following key sufficient condition
\begin{equation} \label{eq:cond}
\{t: t \in \mbox{ the support of } T-\beta_0 X \} \subseteq \{t: t \in \mbox{ the support of } C-\beta_0 X \}
\end{equation}
for the consistency of (\ref{eq:alpha1}). Now we replace $\beta_0$ by its estimator $\hat{\beta}_n$ and want to show the consistency of (\ref{eq:alpha2}). The proof of \citet{StuteWang} for the consistence of the mean survival time estimation makes use of the martingale theory that cannot be directly adopted here due to the dependence between $T-\beta X$ and $C-\beta X$ when $\beta \neq \beta_0$. We shall use the empirical process theory as well as the properties of stochastic integrals with censored data in \citet{LaiYing1988} together with the delicate controlling of the tail fluctuations used by \citet{LaiYing1991} and \citet{Ying1993} to show the desirable result.


First, we introduce the following regularity conditions:
\begin{enumerate}
\item[] {\em Condition 1.} The covariates $X_i$ are i.i.d. random variables with a finite second moment.

\item[] {\em Condition 2.} The error $e_0$'s density $f$ and its derivative $\dot{f}$ are bounded and satisfy
    $$\int_{-\infty}^{\infty} \bigl(\dot{f}(t)/f(t)\bigr)^2f(t)\ dt < \infty.$$

\item[] {\em Condition 3.} The conditional density of $C$ given $X$ is
uniformly bounded for all possible values of $X$, i.e.,
    $$ \displaystyle \sup_{x \in \mathcal{X},\ t\in \mathcal{C}} g_{C|X}(t \mid X=x) < \infty, $$
    where $\mathcal{X}$ and $\mathcal{C}$ denote the support of $X$ and $C$, respectively.

\item[] {\em Condition 4.} The error $e_0$ has a finite second moment, i.e., $Ee_0^2<\infty$.
\end{enumerate}

Condition 1 is different to the common assumption of bounded covariates in \citet{LaiYing1991, Tsiatis, Ying1993}, and many others. Here we only assume finite second moment. Hence, even with a short follow-up time, the support of the censoring time in the residual scale can be extended to infinity provided that the support of $X$ is the real line and $\beta_0 \neq 0$, which yields that the supports of $e_0$ and $C-\beta_0 X$ are equivalent and thus the sufficient condition (\ref{eq:cond}) is satisfied. This requirement is for theoretical justification, whereas in practice, wide support for $X$ works reasonably well. Condition 2 is exactly the same as Condition 2 in \citet{Ying1993}. Condition 3 implies Condition 3 in \citet{Ying1993} as well as Condition (3.5) in \citet{LaiYing1991}. Condition 4 implies Condition 4 in \citet{Ying1993} where $\theta_0 = 2$.

In the following Theorems \ref{thm:consis1} and \ref{thm:consis2}, we omit the constants in front of the rate expressions to further simplify the notation.

\begin{thm} \label{thm:consis1}
Suppose Conditions 1-3 hold, and define
\begin{equation} \label{eq:F(b,t)}
F(\beta,t) = 1 - \exp\biggl\{-\int_{u\leq t} \ \frac{d h^{(0)}(\beta,u)}{h^{(1)}(\beta,u)}\biggr\},
\end{equation}
where $h^{(0)}(\beta,u)$ and $h^{(1)}(\beta,u)$ are given in (\ref{eq:proc1}) and (\ref{eq:proc2}), respectively. Then for every $\varepsilon>0$, with probability 1 we have
\begin{equation}
\sup\bigl\{|\hat{F}_{n,\beta}(t)-F(\beta,t)|:\beta \in \mathcal{B}, H_n^{(1)}(\beta,t)\geq  n^{-\varepsilon}\bigr\}=o(n^{-\frac{1}{2}+3\varepsilon}) \label{eq:thm1.1}
\end{equation}
and
\begin{equation}
\sup\bigl\{|F(\beta,t)-F(t)|:|\beta-\beta_0|\leq n^{-3\varepsilon}, h^{(1)}(\beta,t)\geq n^{-\varepsilon}\bigr\} = O(n^{-\varepsilon}), \label{eq:thm1.2}
\end{equation}
where $\hat{F}_{n,\beta}(t)$ is given in (\ref{eq:KM_b}). Consequently, for every $0<\varepsilon \leq \frac{1}{8}$, with probability 1 we have
\begin{equation}
\sup\big\{|\hat{F}_{n,\beta}(t)-F(t)|:|\beta-\beta_0|\leq n^{-3\varepsilon}, H_n^{(1)}(\beta,t)\geq n^{-\varepsilon}\big\} = O(n^{-\varepsilon}). \label{eq:thm1.3}
\end{equation}
\end{thm}

Introduced by \citet{LaiYing1991}, $F(\beta,t)$ defined in (\ref{eq:F(b,t)}) is an important intermediate quantity. On the one hand, it is the limit of the Kaplan-Meier estimator $\hat{F}_{n,\beta}(t)$ for a fixed $\beta$; on the other hand, it equals to  $F(t)$, the true distribution function of $e_0$, when $\beta$ is replaced by the true slope $\beta_0$ in (\ref{eq:F(b,t)}). The biggest difference between the above Theorem \ref{thm:consis1} and Lemma 2 of \citet{LaiYing1991} is that we do not require bounded covariate support.

\begin{thm} \label{thm:consis2}
Suppose Conditions 1-4 hold, and additionally assume $\beta_0 \ne 0$ and that the support of $X$ is the whole real line, i.e., $f_X(x)>0$ for all $-\infty<x<\infty$. Define
\begin{equation}
T_n = \sup\big\{t: H_n^{(1)}(\beta,t) \geq n^{-\varepsilon}, |\beta-\beta_0|\leq n^{-3\varepsilon} \big\} \label{eq:Tn}
\end{equation}
and let $\hat{F}_{n,\beta}(t) = 1$ for $t>T_n$.
Then for every $0<\varepsilon\leq \frac{1}{8}$, with probability 1 we have
\begin{equation}
\sup\biggl\{\biggl|\int_{-\infty}^{\infty}  t \, d\hat{F}_{n,\beta}(t)-\alpha_0 \biggr|: |\beta-\beta_0| \leq n^{-3\varepsilon} \biggr\}=o(1). \label{eq:thm2}
\end{equation}
\end{thm}

It is clearly seen from Theorem \ref{thm:consis2} that $\hat\alpha_{n, \hat\beta_n}$ given in (\ref{eq:alpha2}) is a consistent estimator of the intercept $\alpha_0$ when $\hat\beta_n$ is consistent with a polynomial convergence rate.
This requires a good estimator of the slope parameter $\beta_0 \ne 0$ under Conditions 1-4 as well as the assumption of unbounded support for $X$. In the next section we show that such an estimator can be obtained by the Gehan-weighted rank based estimating method.

\section{Slope estimation with unbounded covariate support} \label{sec:SlpEst}

Define
\begin{equation}
H_n^{(2)}(\beta,s)=\mathbb{P}_n \{1(\epsilon_\beta \geq s)X\}\quad \mbox{and} \quad h^{(2)}(\beta,s)=P\{1(\epsilon_\beta \geq s)X\}. \label{eq:proc3}
\end{equation}
Then the general rank-based estimating function of \citet{Tsiatis} is given by
\begin{equation} \label{eq:map1}
\mathbb{P}_n \left\{\omega_n(\beta,\epsilon_\beta)\left[ X - \frac{H_n^{(2)}(\beta,\epsilon_\beta)} {H_n^{(1)}(\beta,\epsilon_\beta)}\right]\Delta \right\},
\end{equation}
where $\omega_n(\beta,s)$ is a weight function and $H_n^{(1)}(\beta,s)=\mathbb{P}_n\{1(\epsilon_\beta \geq s)\}$ is defined in (\ref{eq:proc2}). We consider the Gehan weight function $\omega_n(\beta,s) = H_n^{(1)}(\beta,s)$ in (\ref{eq:map1}), which yields the following estimating function
\begin{equation}
\Psi_n\left(\beta,H_n^{(1)},H_n^{(2)}\right)= \mathbb{P}_n\left\{\left[H_n^{(1)}(\beta,\epsilon_\beta)X - H_n^{(2)}(\beta,\epsilon_\beta)\right]\Delta \right\}.  \label{eq:map2}
\end{equation}
It is well-known that the above estimating function is a discrete \citep{Kalbfleisch} and monotone function \citep{Fygenson} of $\beta$, and can be solved by linear programming \citep{Jin2003, LinWeiYing1998} or by a Newton-type algorithm \citep{YuNan2006}.

\subsection{Convergence in probability and asymptotic normality} \label{subsec:ConvProb}

The reason of assuming bounded covariates and/or truncated residual time in the current literature is to bound the denominator $H_n^{(1)}(\beta,\epsilon_\beta)$ in (\ref{eq:map1}) away from zero. Such an issue disappears in (\ref{eq:map2}) when the Gehan weight function is used. Without concerning bounding the denominator away from zero, we can follow the same proofs in \citet{NanEtal2009} to obtain the following results. Details are thus omitted here, but referred to \citet{Ding}.

\begin{proposition} \label{prop:beta_consis1}
Suppose Conditions 1-3 hold. Assume $\beta_0 \in \mathcal{B}$ is the unique root of $\Psi\left(\beta,h^{(1)},h^{(2)}\right)=P\left\{\left[h^{(1)}(\beta,\epsilon_{\beta})X - h^{(2)}(\beta,\epsilon_{\beta})\right]\Delta \right\}$.

(1) The approximate root $\hat{\beta}_n$ satisfying
$$\Psi_n\left(\hat{\beta}_n,H_n^{(1)}(\hat{\beta}_n,\cdot),H_n^{(2)}(\hat{\beta}_n,\cdot)\right)=o_p(1)$$
is a consistent estimator of $\beta_0$.

(2) Suppose that $\Psi\left(\beta,h^{(1)}(\beta, \cdot),h^{(2)}(\beta, \cdot)\right)$ is differentiable with bounded and continuous derivative $\dot{\Psi}_\beta\left(\beta,h^{(1)}(\beta,\cdot), h^{(2)}(\beta,\cdot)\right)$ in a neighborhood of $\beta_0$, and that $\dot{\Psi}_\beta\left(\beta_0,h^{(1)}(\beta_0,\cdot),h^{(2)}(\beta_0,\cdot)\right)$ is nonsingular. Then for an approximate root $\hat{\beta}_n$ satisfying
$$\Psi_n\left(\hat{\beta}_n, H_n^{(1)}(\hat{\beta}_n,\cdot), H_n^{(2)}(\hat{\beta}_n,\cdot)\right) = o_p(n^{-1/2}),$$ we have that $n^{1/2}(\hat{\beta}_n-\beta_0)$ is asymptotically normal with the following asymptotic representation
\begin{equation*}
n^{1/2}(\hat{\beta}_n-\beta_0) = \mathbb{G}_n\{m(\beta_0,\epsilon_0;\Delta,X)\} + o_p(1),
\end{equation*}
where
\begin{eqnarray}
&& \hspace{-0.4in} m(\beta_0,\epsilon_0;\Delta,X) \nonumber  \\\label{eq:beta_normality}
&& \hspace{-0.2in} = \left\{-\dot{\Psi}_\beta\left(\beta_0,h^{(1)}(\beta_0,\cdot),h^{(2)}(\beta_0,\cdot)\right)\right\}^{-1} \biggr\{\left[h^{(1)}(\beta_0,\cdot)X-h^{(2)}(\beta_0,\cdot)\right]\Delta \\
&& \qquad  - \int [1(\epsilon_0 \geq t)X]\ dP_{\epsilon_0,\Delta}(t,1) + \int [1(\epsilon_0\geq t)]x\ dP_{\epsilon_0,\Delta,X}(t,1,x)\biggl\}. \nonumber
\end{eqnarray}
\end{proposition}

The above Proposition \ref{prop:beta_consis1} implies that $|\hat{\beta}_n - \beta_0| =O_p(n^{-3\varepsilon})$ for any $0 < \varepsilon \leq \frac{1}{8}$ with probability approaching 1. Hence we have that $\hat\alpha_{n, \hat\beta_n}$ converges to $\alpha_0$ in probability by Theorem \ref{thm:consis2}.

\subsection{Almost sure convergence with polynomial rate} \label{subsec:Convas}
Following Theorem 5 in \citet{Ying1993}, the almost sure consistency of the Gehan-weighted rank based slope estimator $\hat\beta_n$ with a polynomial rate can also be achieved under the unbounded covariate support assumption, which leads to the strong convergence of the intercept estimator $\hat\alpha_{n,\hat\beta_n}$ from Theorem \ref{thm:consis2}.

\begin{proposition} \label{prop:beta_consis2}
Suppose all the assumptions in Theorem \ref{thm:consis2} hold, and additionally we assume that the tail probability of $X$ satisfies
\begin{equation} \label{eq:tailX}
P(|X|>t) \leq Mt^{\theta}\exp(-\eta t^{\gamma})
\end{equation}
for some constants $M>0$, $|\theta| < \infty$, $\eta>0$, and $\gamma>0$. Then with probability 1 the estimator $\hat{\beta}_n$ satisfying $\Psi_n\left(\hat{\beta}_n, H_n^{(1)}(\hat{\beta}_n,\cdot), H_n^{(2)}(\hat{\beta}_n,\cdot)\right) = o(n^{-1/2})$ converges to $\beta_0$ with a polynomial rate, i.e., $|\hat{\beta}_n-\beta_0| = o(n^{-1/2+\varepsilon})$ almost surely for every $\varepsilon>0$.
\end{proposition}

The exponential tail probability bound (\ref{eq:tailX}) implies Condition 1 in \citet{Ying1993}, which is $\max_{i\leq n}|X_i| = o(n^{\varepsilon})$ almost surely for every $\varepsilon > 0$. This is because for every $t>0$, we have
\begin{eqnarray*}
P\left(\max_{i\leq n}|X_i| > t\right) &=& 1-P\left(\max_{i\leq n}|X_i| \leq t\right) \\
& =& 1-[1-P(|X|>t)]^n  \\
&\leq & 1-[1-Mt^{\theta}\exp(-\eta t^{\gamma})]^n \\
&\leq & nMt^{\theta}\exp(-\eta t^{\gamma}),
\end{eqnarray*}
where the last inequality holds due to the fact that $(1-s)^n \geq 1-ns$ for $0\leq s \leq 1$. Therefore, for every fixed $t>0$ and $\varepsilon>0$,
\begin{eqnarray*}
 \sum_{n=1}^{\infty} P\left(\max_{i\leq n}|X_i| > n^{\varepsilon}t \right) \leq  \sum_{n=1}^{\infty} n M(n^{\varepsilon}t)^{\theta}\exp\{-\eta(n^{\varepsilon}t)^{\gamma}\} < \infty.
\end{eqnarray*}
Then by the Borel-Cantelli lemma,
$P\left(\lim_{n\rightarrow \infty} n^{-\varepsilon}\max_{i\leq n}|X_i| = 0 \right) = 1,$ i.e., $\max_{i \leq n}|X_i| = o(n^{\varepsilon})$ almost surely. As we mentioned earlier, our Conditions 2-4 imply Conditions 2-4 in \citet{Ying1993}. Furthermore, \citet{Ying1993} pointed out that Gehan weights satisfy his Condition 5 and his equation (4.7). Hence the conclusion in Proposition \ref{prop:beta_consis2} follows directly from his equation (4.8) in his Theorem 5. The detailed argument is thus omitted. The exponential tail probability condition holds for many commonly used distributions, for example, normal, weibull, and extreme value distributions.

\section{Simulations}\label{sec:Simu}
\subsection{Intercept estimation}
We conduct extensive simulations to investigate the finite sample performance of the intercept estimation under different scenarios. Failure times are generated from the following model
\begin{equation*}
T = 2 + X_1 + X_2 + \zeta.
\end{equation*}
This is a submodel of \citet{Jin2006, ZengLin2007, DingNan2011} in their simulations. Five different error distributions are considered, which are (a) $\zeta \sim
N(0,0.5^2)$; (b) $\zeta \sim Gumbel(-0.5\gamma, 0.5)$ that has mean zero, where $\gamma$ is the Euler constant; (c) $\zeta \sim Laplace(0,0.5)$; (d) $\zeta \sim Logistic(0,0.5)$; and (e) $\zeta \sim t(0,df=30)$. In each scenario, $X_1$ is Bernoulli with $p=0.5$ and $X_2$ is continuous. We investigate three different distributions of $X_2$: (1) $X_2\sim N(0,1)$; (2) $X_2\sim U(-2,2)$; and (3) $X_2\sim U(-0.5,0)$. The censoring distribution is $C\sim U(0,5) \wedge \tau$, here $\tau$ is a truncation time that reflects the length of follow-up time. We choose $\tau=1.5$ and $\tau=4$ to yield censoring rate ranges $(76\%, 88\%)$ and $(45\%, 52\%)$, respectively. We simulate 1000 runs for each setting, and report the simulation results in Table \ref{table:tab2.1} for two different sample sizes: 100 and 400.

\begin{table}
\caption{Summary of the simulation statistics. The empirical mean (standard deviation) for both the intercept and slope parameters are provided. (a) $\zeta \sim N(0,0.5^2)$; (b) $\zeta \sim Gumbel(-0.5\gamma, 0.5)$; (c) $\zeta \sim Laplace(0,0.5)$; (d) $\zeta \sim Logistic(0,0.5)$; and (e) $\zeta \sim T(0,df=30)$. \dag: $\tau=1.5$ and \ddag: $\tau=4$.}
\label{table:tab2.1}
\begin{center}
\begin{tabular}{lccccccc}
\hline
\hline
Err. & Cen. &  \multicolumn{3}{c}{n = 100 } &  \multicolumn{3}{c}{n = 400 } \\
\cline{3-8}
dist & rate &              $\alpha$  &          $\beta_1$ &         $\beta_2$ &       $\alpha$  &    $\beta_1$ &         $\beta_2$ \\
\hline
\\
\multicolumn{3}{l}{$X_2 \sim N(0,1)$:} & \smallskip \\

(a) & $.83^{\dag}$ 	 & 2.00 (.22) & 1.02 (.27) & 1.01 (.18)	 & 2.00 (.10) & 1.00 (.12) & 1.00 (.08) \\

    & $.51^{\ddag}$  & 2.00 (.09) &	1.00 (.14) & 1.00 (.08)  & 2.00 (.04) &	1.00 (.07) & 1.00 (.04) \smallskip \\

(b) & $.82^{\dag}$ 	&  1.96 (.19) & 1.02 (.24) & 1.01 (.15)  & 1.98 (.10) & 1.00 (.10) & 1.00 (.07) \\

    & $.51^{\ddag}$ &  2.00 (.11) &	1.01 (.14) & 1.00 (.08)  & 2.00 (.05) &	1.00 (.07) & 1.00 (.04) \smallskip  \\

(c) & $.82^{\dag}$	&  1.99 (.28) &	1.04 (.39) & 1.02 (.25)  & 1.99 (.13) &	1.00 (.18) & 1.00 (.11) \\

    & $.51^{\ddag}$ &  2.00 (.12) &	1.00 (.17) & 1.00 (.09)  & 2.00 (.06) &	1.00 (.08) & 1.00 (.04)  \smallskip \\

(d) & $.80^{\dag}$  &  1.97 (.30) &	1.04 (.40) & 1.02 (.26)  & 1.97 (.14) &	1.01 (.18) & 0.99 (.11) \\

    & $.51^{\ddag}$ &  1.99 (.16) &	1.00 (.23) & 1.00 (.12)  & 2.00 (.07) &	1.00 (.11) & 1.00 (.06) \smallskip  \\

(e) & $.78^{\dag}$  &  1.95 (.30) &	1.03 (.39) & 1.03 (.23)  & 1.96 (.15) &	1.02 (.19) & 1.00 (.11) \\

    & $.51^{\ddag}$ &  1.99 (.11) &	1.00 (.26) & 1.01 (.14)  & 1.99 (.05) &	1.00 (.13) & 1.00 (.07) \smallskip \\

\multicolumn{3}{l}{$X_2 \sim U(-2,2)$:} & \smallskip \\

(a) & $.79^{\dag}$  & 2.03 (.24)  &	1.01 (.24) & 1.02 (.18)  & 1.99 (.10) & 1.00 (.11) & 1.00 (.08) \\

    & $.52^{\ddag}$ & 2.01 (.09)  &	1.00 (.14) & 1.00 (.07)  & 2.00 (.04) & 1.00 (.07) & 1.00 (.03) \smallskip  \\

(b) & $.78^{\dag}$  & 1.98 (.21)  &	1.01 (.21) & 1.02 (.16)  & 1.98 (.10) & 0.99 (.10) & 0.99 (.07)\\

    & $.51^{\ddag}$ & 2.00 (.11)  &	1.00 (.14) & 1.00 (.07)  & 2.00 (.05) & 1.00 (.07) & 1.00 (.03) \smallskip \\

(c) & $.78^{\dag}$ 	& 2.02 (.30)  &	1.02 (.32) & 1.04 (.25)  & 1.99 (.14) &	1.00 (.15) & 1.00 (.11) \\

    & $.51^{\ddag}$ & 2.00 (.12)  &	1.00 (.17) & 1.00 (.09)  & 2.00 (.06) &	1.00 (.08) & 1.00 (.04) \smallskip \\

(d) & $.77^{\dag}$  & 1.99 (.30)  &	1.02 (.36) & 1.03 (.25)  & 1.98 (.15) &	1.00 (.17) & 1.00 (.11) \\

    & $.51^{\ddag}$ & 2.00 (.16)  &	1.00 (.23) & 1.00 (.11)  & 2.00 (.08) &	1.00 (.11) & 1.00 (.05) \smallskip \\

(e) & $.76^{\dag}$ 	& 1.96 (.29)  &	1.02 (.36) & 1.03 (.23)  & 1.95 (.14) &	1.01 (.18) & 1.00 (.11) \\

    & $.52^{\ddag}$ & 1.99 (.18)  &	1.00 (.27) & 1.00 (.13)  & 2.00 (.09) &	1.00 (.13) & 1.00 (.06) \smallskip \\

\multicolumn{3}{l}{$X_2 \sim U(-0.5,0)$:} & \smallskip \\

(a) & $.88^{\dag}$  &  1.81 (.28) & 0.74 (.28) & 1.06 (.74)  & 1.80 (.28) & 1.11 (.29) & 1.00 (.34) \\

    & $.45^{\ddag}$ &  2.00 (.14) & 1.00 (.13) & 0.99 (.44)  & 2.00 (.07) & 1.00 (.06) & 1.00 (.22) \smallskip \\

(b) & $.85^{\dag}$  &  1.77 (.21) & 0.78 (.36) & 1.07 (.58)  & 1.75 (.11) & 1.20 (.42) & 1.01 (.29) \\

    & $.45^{\ddag}$ &  1.99 (.16) & 1.00 (.13) & 1.00 (.48)  & 2.00 (.07) & 1.00 (.06) & 1.00 (.22) \smallskip \\

(c) & $.86^{\dag}$  &  1.80 (.40) & 0.95 (.40) & 1.10 (1.18) & 1.75 (.21) & 1.06 (.30) & 1.01 (.52) \\

    & $.45^{\ddag}$ &  1.99 (.18) & 1.00 (.15) & 0.99 (.56)  & 2.00 (.08) & 1.00 (.06) & 1.00 (.26) \smallskip \\

(d) & $.81^{\dag}$  &  1.68 (.42) & 1.04 (.39) & 1.04 (1.18) & 1.66 (.19) & 1.02 (.21) & 1.01 (.51) \\

    & $.46^{\ddag}$ &  1.99 (.24) & 1.00 (.21) & 0.99 (.76)  & 2.00 (.11) & 1.00 (.11) & 1.00 (.35) \smallskip \\

(e) & $.79^{\dag}$  &  1.61 (.40) & 1.03 (.38) & 1.03 (1.15) & 1.59 (.20) & 1.01 (.18) & 1.01 (.52) \\

    & $.46^{\ddag}$ &  1.98 (.28) & 1.00 (.25) & 0.98 (.87)  & 2.00 (.14) & 0.99 (.13) & 1.01 (.42) \\
\hline
\end{tabular}
\end{center}
\end{table}

The first covariate setting corresponds to the unbounded covariate support. It is clearly seen that the bias of the intercept estimator is minimal even with the shorter truncation time $\tau=1.5$ for all error distributions. The bias is also very small in the second covariate setting, where the support of $X_2$ is bounded, but wide. The bias becomes noticeable when the support of $X_2$ gets narrower in the third setting with truncation time $\tau=1.5$. With the longer truncation time $\tau=4$,  which is close to the setting of \citet{LaiYing1991} who assumed wider censoring time support, the bias of the intercept estimator is negligible for all error distributions and covariate supports. The bias for the slope estimators is minimal for most of the simulation settings except for the binary covariate $X_1$ under the third setting ($X_2 \sim U(-0.5,0)$) with Normal and Gumbel errors when the follow-up time is short. This is possibly because when the censoring rate is very high ($\ge 85\%$), the probability to observe a non-zero value of $X_1$ under the uncensored case is very small (about $7.5\%$ or less), therefore the estimation for $\beta_1$ did not perform well in this case.

For the shorter follow-up setting with $\tau=1.5$, Figure \ref{Figure:Fig2.1} displays the Kaplan-Meier curves of the estimated residual survival time $T_i-\hat{\beta}_{n,1}X_{i,1}-\hat{\beta}_{n,2}X_{i,2}$ under five error distributions, where each curve is obtained from a sample with size $n=400$. From left to right, the three panels correspond to $X_2 \sim N(0,1)$, $X_2 \sim U(-2,2)$, and $X_2 \sim U(-0.5,0)$, respectively. It is clearly seen that whenever the survival curve is close to zero at the right tail, the corresponding intercept estimator in Table \ref{table:tab2.1} has minimal bias. We notice from our intensive simulations that a satisfactory intercept estimator (bias $< 5\%$) can be obtained when the right tail of the Kaplan-Meier curve goes below 0.15. This provides a practical rule of thumb for getting a sense of adequacy of the intercept estimation.

\begin{figure}
\begin{center}
\includegraphics[scale=0.7]{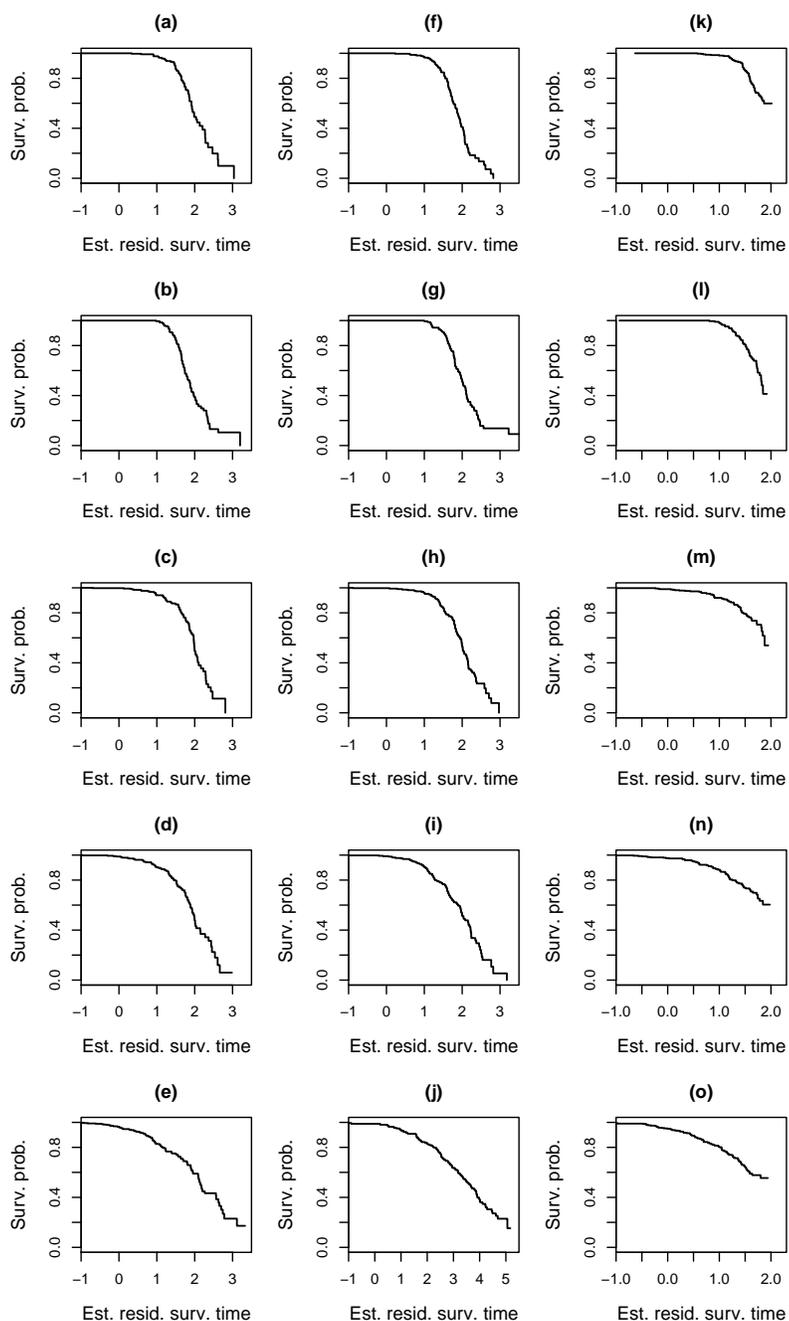}
\caption{K-M plots of the estimated residual survival time ($T-\hat{\beta}_{n,1}X_1-\hat{\beta}_{n,2}X_2$) under $\tau=1.5$. Each column corresponds to one of the five error distributions given in Table \ref{table:tab2.1}. (a)-(e): $X_2\sim N(0,1)$; (f)-(j): $X_2\sim U(-2,2)$; (k)-(o): $X_2\sim U(-0.5,0)$.} \label{Figure:Fig2.1}
\end{center}
\end{figure}

\subsection{Failure time prediction}

We also compare the survival time prediction accuracy of the linear model with the Cox model via simulations. In order to have a fair comparison, we generate data from the following model:
\begin{equation} \label{model:AFTvsCox}
T = X + e_0,
\end{equation}
where $e_0$ follows the standard extreme value distribution with $F(t) = 1-e^{-e^t}$. It is well known that such a model setting fits both the linear regression model and the Cox model. In model (\ref{model:AFTvsCox}), we have $\alpha_0 = E e_0 = -\gamma$ and $\beta_1=1$, where $\gamma$ is the Euler constant. Note that we only use a single covariate for illustrative purpose.

We generate censoring time from $C \sim U(-3,3)\wedge \tau$, where $\tau$ is a fixed truncation time taking different values in order to generate different censoring rates. As in the first simulation study, covariate $X$ is generated from three distributions: $N(0,1)$, $U(-2,2)$ and $U(-1,1)$ to represent three scenarios of the covariate support (unbounded; bounded but wide; bounded but narrow). For each simulation setting, two independent data sets of equal size are generated, namely the training set and the test set, at each simulation run. Both the linear model and the Cox model are fitted using the training set, and survival times are predicted for the test set using the fitted models.

For the linear model, the predicted survival time for subject $i$ in the test set with covariate $X_i^*$ is calculated as $\hat{T}_i^{LR}=\hat{E}(T_i|X_i^*)=\hat{\alpha}_n+\hat{\beta}_n^{LR} X_i^*$,
where $\hat{\beta}_n^{LR}$ is solved by the Gehan-weighted rank based estimating equation and $\hat{\alpha}_n$ is estimated from (\ref{eq:alpha2}). For the Cox model, the predicted survival time is calculated by $$\hat{T}_i^{Cox} = \int t \ d\left[1-\exp \left\{-\hat{\Lambda}_{0,n}(t)e^{\hat{\beta}_n^{Cox}X_i^*}\right\}\right],$$ where $\hat{\Lambda}_{0,n}(t)$ is the Breslow estimator of the baseline cumulative hazard function $\Lambda_0(t)$, whereas $\hat{\beta}_n^{Cox}$ is the partial likelihood estimator. We use the following measure to determine the prediction accuracy:
\begin{equation} \label{eq:MSE}
MSE_p = \frac{1}{n}\sum_{i=1}^{n}(T_i^*-\hat{T}_i)^2,
\end{equation}
where $\hat{T}_i$ is either $\hat{T}_i^{LR}$ or $\hat{T}_i^{Cox}$ depending on which model is used and $T_i^*$ is the true survival time for the $i$th subject in the test set. Two sample sizes are considered:  $n=200$ and $n=2000$, and 1000 runs are conducted for each simulation setting. The results are summarized in Table \ref{table:tab2.2}. For each scenario, we calculate the relative prediction accuracy to the case without censoring, i.e., the ratio of the empirical mean $MSE_p$ under no censoring to that under each corresponding censored case, in addition to  reporting the empirical mean of $MSE_p$ (given in parentheses). Note that $\tau \geq 3$ implies no truncation. The $MSE_p$ obtained from ordinary least squares (OLS) is also listed for each no-censoring scenario.

\begin{table}
\caption{Comparison of prediction accuracy.  Relative prediction accuracy to the case without censoring is listed. Empirical mean of $MSE_p$ is given in parentheses. $MSE_p$ obtained from ordinary least squares (OLS) is also listed. (a): $X \sim N(0,1)$; (b): $X \sim U(-2,2)$; (c): $X \sim U(-1,1)$.}
\label{table:tab2.2}
\begin{center}
\begin{tabular}{ccccccc}
\hline
\hline
           &            &     Cen. &                  \multicolumn{ 4}{c}{Sample Size} \\

         $X$ &    $\tau$ &       rate & \multicolumn{ 2}{c}{n = 200} & \multicolumn{ 2}{c}{n = 2000} \\
\cline{4-7}
           &            &            &     Linear &        Cox &     Linear &        Cox \\
\hline
\\
(a)        &         -2 &        0.86 & 0.86 (1.95) & 0.33 (5.08) & 0.98 (1.68) & 0.33 (4.98) \\

           &         -1 &        0.72 & 0.97 (1.72) & 0.58 (2.90) & 0.99 (1.66) & 0.58 (2.86) \\

           &          0 &        0.55 & 0.99 (1.69) & 0.84 (2.00) & 1.00 (1.65) & 0.84 (1.96) \\

           &          1 &        0.44 & 1.00 (1.67) & 0.97 (1.72) & 1.00 (1.64) & 0.97 (1.70) \\

           &         $ \geq 3$ &        0.00 & 1.00 (1.67) & 1.00 (1.67)  & 1.00 (1.65) & 1.00 (1.65)  \\

           & \multicolumn{ 2}{c}{OLS} & \multicolumn{ 2}{c}{(1.67)} & \multicolumn{ 2}{c}{(1.65)} \smallskip \\

(b) 			 &         -2 &        0.82 & 0.85 (1.93) & 0.31 (5.38) & 0.96 (1.71) & 0.31 (5.28) \\

           &         -1 &        0.67 & 0.96 (1.71) & 0.53 (3.12) & 1.00 (1.65) & 0.54 (3.08) \\

           &          0 &        0.54 & 0.99 (1.67) & 0.80 (2.07) & 1.00 (1.65) & 0.80 (2.05) \\

           &          1 &        0.46 & 0.99 (1.66) & 0.96 (1.72) & 1.00 (1.65) & 0.96 (1.71) \\

           &         $ \geq 3$ &        0.00 & 1.00 (1.65) & 1.00 (1.65)  & 1.00 (1.65) & 1.00 (1.65)  \\

           & \multicolumn{ 2}{c}{OLS} & \multicolumn{ 2}{c}{(1.65) } & \multicolumn{ 2}{c}{(1.65)} \smallskip \\

(c)        &         -2 &       0.86 & 0.68 (2.41) & 0.37 (4.51) & 0.74 (2.24) & 0.38 (4.38) \\

           &         -1 &       0.72 & 0.94 (1.75) & 0.67 (2.47) & 0.97 (1.70) & 0.67 (2.45) \\

           &          0 &       0.55 & 0.99 (1.66) & 0.93 (1.78) & 1.00 (1.65) & 0.93 (1.77) \\

           &          1 &       0.44 & 1.00 (1.65) & 1.00 (1.65) & 1.00 (1.65) & 1.00 (1.65) \\

           &         $ \geq 3$ &       0.00 & 1.00 (1.65) & 1.00 (1.65) & 1.00 (1.65) & 1.00 (1.65) \\

           & \multicolumn{ 2}{c}{OLS} & \multicolumn{ 2}{c}{(1.65)} & \multicolumn{ 2}{c}{(1.65)} \\

\hline
\end{tabular}
\end{center}
\end{table}

From Table \ref{table:tab2.2} we see that the linear model is much less sensitive to the truncation time, especially for wide covariate support, e.g., $X \sim N(0,1)$ and $X \sim U(-2,2)$, where the linear model yields almost perfect prediction error regardless of truncation time. The linear model performs uniformly better than the Cox model. The Cox model does extremely poorly in cases with heavy truncation. This is not surprising because the baseline hazard function in the Cox model is not estimable after the last observation time ($\leq \tau$) in the training set. The convention is to set the failure time distribution function to be 1 after that time point. This introduces bias when predicting the survival time and obviously, the bias becomes more severe when the follow-up time is shorter. The difference between the two models is clearly seen from Figure \ref{Figure:Fig2.3}. The two models perform equally well when there is no censoring (see Table \ref{table:tab2.2}).

\begin{figure}
\begin{center}
\includegraphics[scale=0.8]{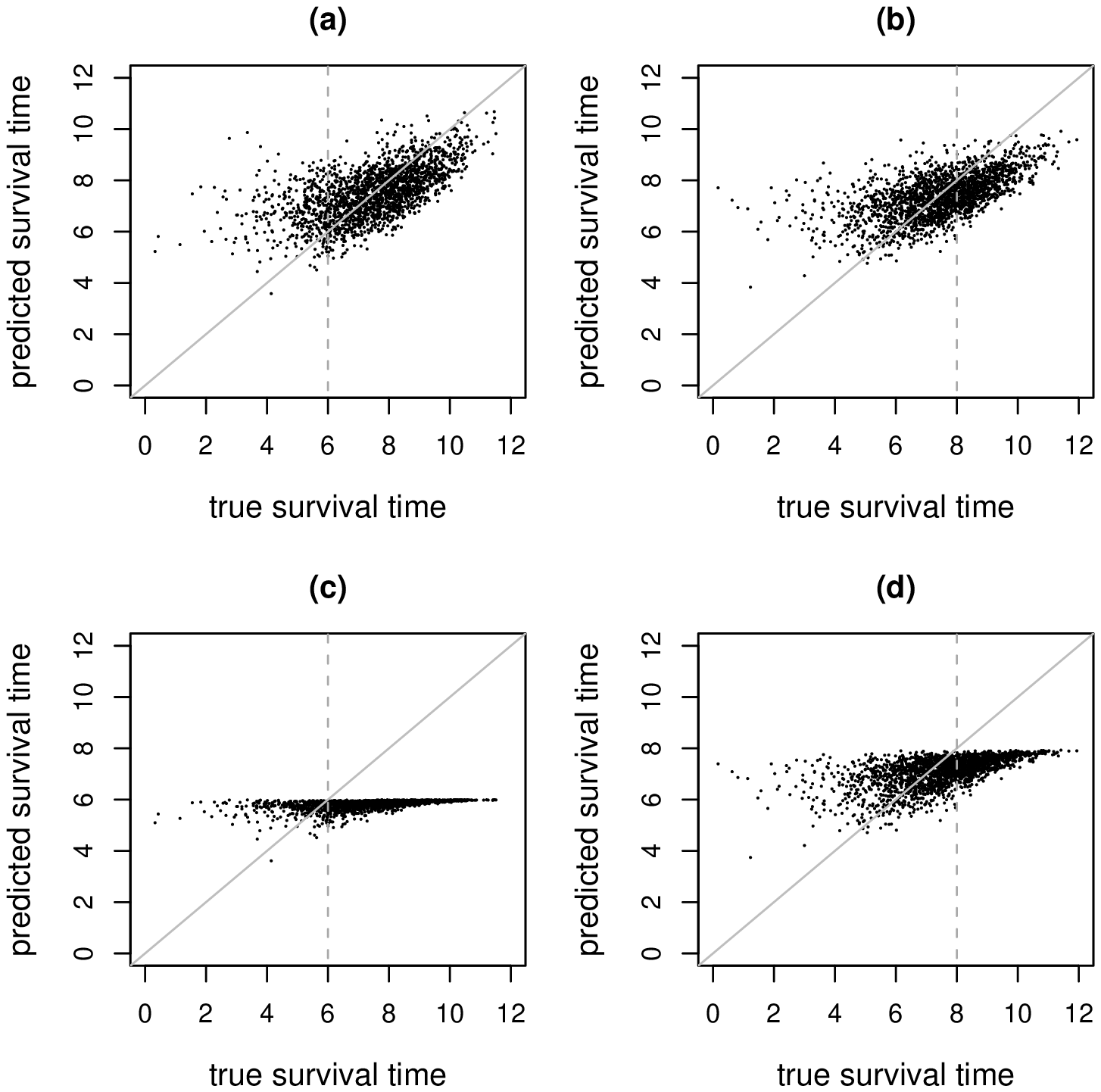}
\caption{The predicted survival time versus the true survival time for data generated from model (\ref{model:AFTvsCox}) with $X\sim N(0,1)$ and $n=2000$. (a): linear model with $\tau=-2$; (b): linear model with $\tau=0$; (c): Cox model with $\tau=-2$; and (d): Cox model with $\tau=0$. Constant $8$ was added to shift all the simulated survival times to positive values.} \label{Figure:Fig2.3}
\end{center}
\end{figure}

\section{A real data example} \label{sec:Exp}

We consider the well-known Mayo primary biliary cirrhosis (PBC) study as an illustrative example \citep[app. D.1]{FlemingHarrington}. The data contain information about the survival time and prognostic factors for 418 patients. \citet{Jin2003} and \citet{Jin2006} fitted the accelerated failure time model with five covariates, namely age, log(albumin), log(bilirubin), edema, and log(protime). They estimated slope parameters for those covariates using rank-based and least squares methods. We consider the same model. The slope parameters are estimated by Gehan weighted rank based approach and the intercept estimator is obtained by (\ref{eq:alpha2}). The estimated coefficients for the five prognostic factors are $(-0.025, 1.498, -0.554, -0.904, -2.822)$ with estimated standard errors  $(0.005, 0.479, 0.052, 0.234, 0.923)$, which are similar to those reported in \citet{Jin2003}. The intercept estimator is 8.692. The right tail of the Kaplan-Meier curve of the residual survival time almost touches zero (see Figure \ref{Figure:Fig2.4}a), which indicates a valid intercept estimation from  (\ref{eq:alpha2}) for the PBC data.

We then perform the leave-one-out cross-validation to check the prediction performance of the model.
Figure \ref{Figure:Fig2.4}b shows the predicted survival time against the observed time in the logarithm scale.
The circles correspond to the patients who failed and the triangles correspond to the patients who were censored.
The figure suggests that the accelerated failure time model provides a reasonable prediction of the survival time for this dataset with most of the censored subjects having predicted survival times above the 45-degree line, except for a few subjects who might be outliers. For example, subject 87 (circled in Figure \ref{Figure:Fig2.4}b) was a 37 year old woman with quite good prognostic status: no edema, good albumin (4.4), low bilirubin (1.1) and moderate protime (10.7). Yet she survived for no longer than roughly half a year. Subject 293, on the other hand, was a 57 year old woman with poor prognostic status. In spite of low albumin (2.98), high bilirubin (8.5) and protime (12.3), and edema resistent to diuretics, she remains alive after more than 3.5 years. This same subject was also detected as an outlier in the residual plot for the covariate edema from a Cox model for the same data \citep[p. 184]{FlemingHarrington}.

\begin{figure}
\begin{center}
\includegraphics[scale=0.6]{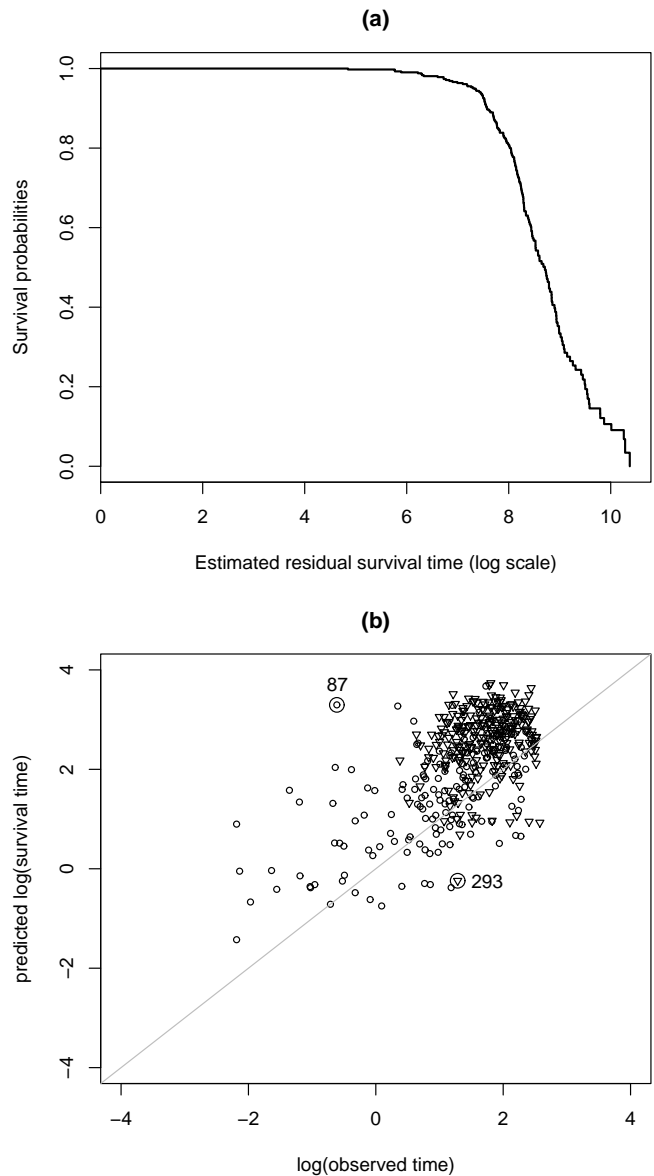}
\caption{(a): K-M plot of the estimated residual survival time for the PBC data. (b): Predicted survival time versus the observed time points (in the logarithm scale) for the PBC data. Circle: individual who failed; triangle: individual who was censored. Subject 87 and 293 are two potential outliers.} \label{Figure:Fig2.4}
\end{center}
\end{figure}

\section{Concluding remarks} \label{sec:Con}
In practice, the mean survival time can be well estimated when the follow-up time is long or the covariate range is wide (even with a short follow-up time). The first situation corresponds to the well-known technical condition that the support of the censoring time contains the support of the survival time and the second situation corresponds to the technical condition of unbounded covariate support that we have established in this article.

Model checking is very important in data analysis. For the linear model, one can follow the method developed for the Cox model by visualizing the cumulative sums of the martingale-based residuals to assess how unusual the observed residual patterns would be, see e.g. \citet{LinRobinsWei1996, LinWei1993}.

Bias and variance trade-off plays an important role in assessing prediction errors, which requires knowing the asymptotic joint distribution of both the intercept and slope parameter estimators. We do not pursue it here. We also want to point out that any prediction beyond the follow-up time needs to be interpreted cautiously because it lacks empirical verification without obtaining new data with extended follow-up.

The asymptotic distribution of the intercept estimator is still unknown. However, a trimmed mean can be estimated at ${n}^{1/2}$-rate with an asymptotically normal distribution. We refer to \citet{Ding} for details.
%

\section{Proofs of technical results}\label{sec:Proofs}
We prove Theorems \ref{thm:consis1}-\ref{thm:consis2} in this section. Firstly, we provide several lemmas that will be used in the proofs.

\subsection{Technical lemmas}
\begin{lem} \label{lem:lem1}
For every $\varepsilon >0$, with probability 1 we have
\begin{equation*}
\sup_{\beta \in \mathcal{B}, -\infty<s<\infty}n^{1/2}|H_n^{(k)}(\beta,s)-h^{(k)}(\beta,s)|
= o(n^{\varepsilon}),
\end{equation*}
where $H_n^{(k)}(\beta,s)$ and $h^{(k)}(\beta,s)$, $k=0,1$, are defined in (\ref{eq:proc1}) and (\ref{eq:proc2}) respectively.
\end{lem}

\begin{proof}
We apply the empirical process theory to prove this result. Since the class of indicator functions of half spaces is a VC-class, see e.g. Exercise 9 on page 151 and Exercise 14 on page 152 in \citet{vanWellner1996}, and thus a Donsker class, the sets of functions $\mathcal{F}_0 = \{1(\epsilon_\beta \leq s, \Delta=1)\}$=$\{\Delta 1(\epsilon_\beta \leq s)\}$ and $\mathcal{F}_1 = \{1(\epsilon_\beta \geq s)\}$ are both Donsker classes. Let $\bar{\mathcal{F}}_k$ be the closure of $\mathcal{F}_k$, $k=0,1$, respectively. Then $H_n^{(k)}(\beta,s)$ and $h^{(k)}(\beta,s)$ are in the convex hull of $\bar{\mathcal{F}}_k$, $k=0,1$, and thus belong to Donsker classes (see e.g. Theorems 2.10.2 and 2.10.3 in \citet{vanWellner1996}). Hence by their Theorem 2.6.7 and Theorem 2.14.9, it follows that for every $t>0$,
$$ P\biggl(\sup_{\beta \in \mathcal{B}, -\infty<s<\infty}n^{1/2}|H_n^{(k)}(\beta,s)-h^{(k)}(\beta,s)|>t\biggr) \leq Mt^Ve^{-2t^2}, $$
where $M>0$ is a constant and $V=2V(\mathcal{F})-2$ with $V(\mathcal{F})$ being the index of the VC-class $\mathcal{F}$, which is 4 in this case for one-dimensional $\beta_0$, hence $V=6$. When $\beta_0 \in \mathbb{R}^d$ for a fixed $d$, the index of the VC-class is $V(\mathcal{F})=d+3$ and the following argument still holds. Then for any $\varepsilon>0$, let
$$\displaystyle A_{n,\varepsilon}=\sup_{\beta \in \mathcal{B}, -\infty<s<\infty} n^{1/2-\varepsilon}|H_n^{(k)}(\beta,s)-h^{(k)}(\beta,s)|.$$ Since $t^V \leq e^{1.5t^2}$ for large enough $t>0$ and a fixed $V>0$, then
$$ \displaystyle \sum_{n=1}^{\infty}P(|A_{n,\varepsilon}-0|>t) \leq M\sum_{n=1}^{\infty}\exp\{-0.5(n^{\varepsilon}t)^2\} < \infty .$$
By the Borel-Cantelli lemma we have $\displaystyle P\big(\lim_{n \rightarrow \infty} A_{n,\varepsilon}=0\big)=1$. We then have obtained the desired result.
\end{proof}

\begin{lem} \label{lem:lem2}
Assume Conditions 1-3 hold, then for every $\varepsilon \geq 0 $ we have
\begin{equation*}
\sup_{|\beta-\beta'|+|s-s'|\leq n^{-\varepsilon}}|h^{(k)}(\beta,s)-h^{(k)}(\beta',s')| = O(n^{-\varepsilon}),
\end{equation*}
where $h^{(k)}(\beta,s)$, $k=0,1$ and $2$, are defined in (\ref{eq:proc1}), (\ref{eq:proc2}) and (\ref{eq:proc3}) respectively.
\end{lem}

\begin{proof}
Since $e_0=T-\beta_0 X$ is independent of $(X,C)$,
the joint density function of $(T,C,X)$ can then be decomposed as
$$ f_{T,C,X}(t,c,x)=f_{e_0,C,X}(t-\beta_0 x,c,x)=f(t-\beta_0 x)f_{C,X}(c,x) $$
where $f$ is the density of $e_0$. So
$$ f(t-\beta_0 x) = f_{T|C, X}(t|C=c,X=x) = f_{T|X}(t|X=x).$$
Then the joint density function of $(Y,\Delta,X)$ follows
\begin{eqnarray*}
&& \hspace{-0.2in}  f_{Y,\Delta,X}(y,\delta,x)\\
&& =f(y-\beta_0 x)^{\delta}\bar{F}(y-\beta_0 x)^{1-\delta} g_{C|X}(y|X=x)^{1-\delta}\bar{G}_{C|X}(y|X=x)^{\delta}f_X(x),
\end{eqnarray*}
where $\bar{F}(\cdot)=1-F(\cdot)$ and $\bar{G}_{C|X}(\cdot|X=x)=1-G_{C|X}(\cdot|X=x)$.

For $h^{(0)}(\beta,s)$, the joint sub-density function of $(Y,\Delta=1,X)$ can be written as
$ f_{Y,\Delta,X}(y,1,x)=f(y-\beta_0 x)\bar{G}_{C|X}(y|X=x)f_X(x)$. So
\begin{eqnarray*}
h^{(0)}(\beta,s) \!\!\!&=& \!\!\! P\{1(\epsilon_\beta \leq s, \Delta=1)\}\\
&=& \!\!\! \int_{\mathcal{X}} \biggl\{ \int_{-\infty}^s f(u+(\beta-\beta_0) x)\bar{G}_{C|X}(u+\beta x|X=x)\, du \biggr\}f_X(x)\, dx.
\end{eqnarray*}
Then for any $\beta,\beta' \in \mathcal{B}$ and $-\infty<s<\infty$, by the mean value theorem, there exists a value $\tilde{\beta}$ between $\beta$ and $\beta'$ such that
\begin{eqnarray*}
&& \hspace{-0.2in} |h^{(0)}(\beta,s)-h^{(0)}(\beta',s)| \\
&& = \biggl|\int_{\mathcal{X}}\biggl\{\int_{-\infty}^s \bigl[\dot{f}(u+(\tilde{\beta}-\beta_0)x) \bar{G}_{C|X}(u+\tilde{\beta} x|X=x) \\
&& \qquad -\ f(u+(\tilde{\beta}-\beta_0)x)g_{C|X}(u+\tilde{\beta}x|X=x)\bigr](\beta-\beta')x \ du\biggr\}f_X(x)\ dx\biggr|\\
&&  \leq  |\beta-\beta'| \int_{\mathcal{X}}\biggl\{\int_{-\infty}^s \bigl|\dot{f}(u+(\tilde{\beta}-\beta_0)x)\bar{G}_{C| X}(u+\tilde{\beta}x|X=x) \\
&& \qquad -\ f(u+(\tilde{\beta}-\beta_0)x)g_{C|X}(u+\tilde{\beta}x|X=x)\bigr| \, du \biggr\}|x|f_X(x)\, dx \\
&&  \leq C_1|\beta-\beta'| \int_{\mathcal{X}}\biggl\{\int_{-\infty}^{\infty}\{|\dot{f}(u)|+ f(u)\}\, du \biggr\}|x|f_X(x)\, dx \\
&&  \leq C_1C_2|\beta-\beta'|\int_{\mathcal{X}}|x|f_X(x)\, dx,
\end{eqnarray*}
where the second inequality holds for some finite constant $C_1\geq 1$ such that $g_{C|X}(\cdot|X=x) \leq C_1$ uniformly, which is guaranteed by Condition 3; and the third inequality holds by
Condition 2 and the following Cauchy-Schwartz inequality
\begin{eqnarray*}
\biggl\{\int_{-\infty}^{\infty} |\dot{f}(u)|\, du\biggr\}^2 &\leq& \int_{-\infty}^{\infty} \biggl(\frac{|\dot{f}(u)|}{\sqrt{f(u)}}\biggr)^2\ du \cdot \int_{-\infty}^{\infty} \bigl(\sqrt{f(u)}\bigr)^2\, du \\
&=& \biggl\{\int_{-\infty}^{\infty} \biggl(\frac{\dot{f}(u)}{f(u)}\biggr)^2f(u)\, du\biggr\} \cdot 1 <\infty
\end{eqnarray*}
such that
$$\int_{-\infty}^{\infty}\{|\dot{f}(u)|+f(u)\}\, du = \int_{-\infty}^{\infty}|\dot{f}(u)|\, du + 1 \leq C_2$$ for a constant $C_2 < \infty$. Therefore, by Condition 1 that $X$ has a finite second moment and thus a finite first moment, it follows that $$|h^{(0)}(\beta,s)-h^{(0)}(\beta',s)| \leq K_1|\beta-\beta'|$$
for a constant $K_1 < \infty$.

Moreover, for any $\beta \in \mathcal{B}$ and $-\infty<s,s'<\infty$, we have
\begin{eqnarray*}
&& \hspace{-0.2in} |h^{(0)}(\beta,s)-h^{(0)}(\beta,s')| \\
&& = \biggl|\int_{\mathcal{X}}\biggl\{\int_s^{s'}f(u+(\beta-\beta_0)x)\bar{G}_{C|X}(u+\beta x|X=x)\ du\biggr\}f_X(x)\ dx \biggr| \\
&& \leq C_3|s-s'|,
\end{eqnarray*}
where $C_3$ is a constant such that $f(\cdot)\leq C_3$, which is guaranteed by Condition 2. Hence, for any $\beta,\beta' \in \mathcal{B}$ and $-\infty<s,s'<\infty$, it follows that
$$\sup_{|\beta-\beta'|+|s-s'|\leq n^{-\varepsilon}}|h^{(0)}(\beta,s)-h^{(0)}(\beta',s')| = O(n^{-\varepsilon}).$$

For $h^{(1)}(\beta,s)$, it is easy to obtain that
\begin{equation*} \label{eq:h1_cond}
P\{1(\epsilon_\beta \geq s)|X=x\} = \bar{F}(s+(\beta-\beta_0)x)\bar{G}_{C|X}(s+\beta x|X=x).
\end{equation*}
Then for any $\beta,\beta' \in \mathcal{B}$ and $-\infty<s<\infty$, by the mean value theorem, there exists a value  $\tilde{\beta}$ between $\beta$ and $\beta'$ such that
\begin{eqnarray*}
&& \hspace{-0.2in} |h^{(1)}(\beta,s) - h^{(1)}(\beta',s)|
\\
&&   = \biggl|\int_{\mathcal{X}} \bigl\{\bar{F}(s+(\beta-\beta_0)x) \bar{G}_{C|X}(s+\beta x|X=x) \\
&& \qquad -\ \bar{F}(s+(\beta'-\beta_0)x) \bar{G}_{C|X}(s+\beta'x|X=x)\bigr\} f_X(x)\ dx \biggr| \\
&&   = \biggl|\int_{\mathcal{X}}\bigl\{-f(s+(\tilde{\beta}-\beta_0)x) \bar{G}_{C|X}(s+\tilde{\beta}x|X=x) \\
&& \qquad -\ \bar{F}(s+(\tilde{\beta}-\beta_0)x) g_{C|X}(s+\tilde{\beta}x|X=x) \bigr\} (\beta-\beta')x f_{X}(x)\, dx \biggr| \\
&&  \leq |\beta-\beta'|\int_{\mathcal{X}}\bigl\{f(s+(\tilde{\beta}-\beta_0)x)+g_{C|X}(s+\tilde{\beta}x|X=x)\bigr\} |x|f_X(x)\, dx \\
&&  \leq (C_1+C_3)|\beta-\beta'|\int_{\mathcal{X}}|x|f_X(x)\, dx \\
&& = K_2|\beta-\beta'|
\end{eqnarray*}
for some constant $K_2=(C_1+C_3)E|X| < \infty$, where $C_1$ and $C_3$ are
defined before.
Moreover, for any $\beta \in \mathcal{B}$ and $-\infty<s,s'<\infty$, by the mean value theorem, there exists a value $\tilde{s}$ between $s$ and $s'$ such that
\begin{eqnarray*}
&& \hspace{-0.2in} |h^{(1)}(\beta,s)-h^{(1)}(\beta,s')| \\
&&  = \biggl|\int_{\mathcal{X}}\bigl\{-f(\tilde{s}+(\beta-\beta_0)x) \bar{G}_{C|X}(\tilde{s}+\beta x|X=x)\\
&& \qquad -\ \bar{F}(\tilde{s}+(\beta-\beta_0)x)g_{C|X}(\tilde{s}+\beta x|X=x)\bigr\}(s-s')f_{X}(x)\ dx \biggr| \\
&&  \leq |s-s'|\int_{\mathcal{X}}\bigl\{f(\tilde{s}+(\beta-\beta_0)x)+g_{C|X}(\tilde{s}+\beta x|X=x)\bigr\} f_X(x)\ dx \\
&&  \leq (C_1+C_3)|s-s'|.
\end{eqnarray*}
Hence, for any $\varepsilon>0$, we have $$\sup_{|\beta-\beta'|+|s-s'|\leq n^{-\varepsilon}}|h^{(1)}(\beta,s)-h^{(1)}(\beta',s')| = O(n^{-\varepsilon}).$$

Finally for $h^{(2)}(\beta,s)$, by using the similar argument to that for $h^{(1)}(\beta,s)$, we can easily obtain that
\begin{eqnarray*}
|h^{(2)}(\beta,s)-h^{(2)}(\beta',s)|
\leq (C_1+C_3)|\beta-\beta'|\int_{\mathcal{X}}x^2f_X(x)\ dx = K_3|\beta-\beta'|
\end{eqnarray*}
and
\begin{eqnarray*}
|h^{(2)}(\beta,s)-h^{(2)}(\beta,s')|
 \leq (C_1+C_3)|s-s'|\int_{\mathcal{X}}|x|f_X(x)\ dx = K_2|s-s'|,
\end{eqnarray*}
where $K_3 = (C_1+C_3)EX^2<\infty$. Therefore, for any $\varepsilon>0$, we have
$$\sup_{|\beta-\beta'|+|s-s'|\leq n^{-\varepsilon}}|h^{(2)}(\beta,s)-h^{(2)}(\beta',s')| = O(n^{-\varepsilon}).$$
Thus, we have proved Lemma \ref{lem:lem2}.
\end{proof}

\begin{lem} \label{lem:lem3}
Let $U_n(\beta,s)$ be random variables for which there exist non-random Borel functions $u_n(\beta,s)$ such that for every $\varepsilon >0$,
\begin{enumerate}
\item[(A1)]
$\displaystyle \sup_{\beta \in \mathcal{B}, -\infty<s<\infty}|U_n(\beta,s)-u_n(\beta,s)|= o(n^{-1/2+\varepsilon})$ almost surely.
\item[(A2)] $U_n(\beta,s)$ has a bounded variation in $s$ uniformly on $\mathcal{B}$, that is,
$$\sup_{\beta \in \mathcal{B}}\int_{s=-\infty}^{\infty}\ |dU_n(\beta,s)|=O(1) \ \mbox{almost surely.}$$
\item[(A3)] $u_n$ satisfies $$\sup_{\beta\in \mathcal{B}, -\infty<s<\infty}|u_n(\beta,s)|= O(1).$$
\end{enumerate}
Then under Conditions 1-3, for every $0 < \varepsilon \leq 1/2$, with probability 1 we have
\begin{eqnarray*} \label{eq:lemma3}
&& \sup_{\beta \in \mathcal{B},
-\infty<y<\infty}\biggl|\int_{s=-\infty}^{y}U_n(\beta,s)\ dH_n^{(0)}(\beta,s)-
\int_{s=-\infty}^{y}u_n(\beta,s)\ dh^{(0)}(\beta,s)\biggr| \\ && \qquad = o(n^{-1/2+\varepsilon}).
\end{eqnarray*}
\end{lem}

\begin{proof}
By the triangle inequality and integration by parts, we have
\begin{eqnarray*}
& & \hspace{-0.3in} \biggl|\int_{s=-\infty}^{y}U_n(\beta,s)\ dH_n^{(0)}(\beta,s)-
\int_{s=-\infty}^{y}u_n(\beta,s)\ dh^{(0)}(\beta,s)\biggr| \\
&&  \leq \int_{s=-\infty}^{y}|U_n(\beta,s)-u_n(\beta,s)|\ dh^{(0)}(\beta,s) \\
&& \qquad + \ |U_n(\beta,y)\bigl(H_n^{(0)}(\beta,y)-h^{(0)}(\beta,y)\bigr)| \\
& & \qquad  + \ \int_{s=-\infty}^{y}|H_n^{(0)}(\beta,s)-h^{(0)}(\beta,s)|\ |dU_n(\beta,s)|.
\end{eqnarray*}
Then it is easy to see that each term on the right hand side of the above inequality is $o(n^{-1/2+\varepsilon})$ almost surely under (A1)-(A3) and Lemma \ref{lem:lem1}.
\end{proof}

\subsection{Proof of Theorem \ref{thm:consis1}}
By the first order Taylor expansion of function $\log(1-x)$, for large $n$ we have
\begin{eqnarray*}
\hat{F}_{n,\beta}(t) &=& 1-\exp\biggl\{\sum_{i:\epsilon_{\beta,i}\leq t}\log \biggl(1-\frac{\Delta_i/n}{H_n^{(1)}(\beta,\epsilon_{\beta,i})}\biggr)\biggr\} \\
&=& 1-\exp\biggl\{-\int_{u\leq t}\ \frac{dH_n^{(0)}(\beta,u)}{H_n^{(1)}(\beta,u)}-\sum_{i:\epsilon_{\beta,i}\leq t}O\bigl(\{nH_n^{(1)}(\beta,\epsilon_{\beta,i})\}^{-2}\bigr)\bigg\}.
\end{eqnarray*}
Then by the mean value theorem and the fact that $e^x \leq 1$ for any $x\leq 0$, it follows that
\begin{eqnarray*}
&& \hspace{-0.2in} |\hat{F}_{n,\beta}(t)-F(\beta,t)| \\
&&\!\!\!= \biggl|\exp\biggl\{-\int_{-\infty}^{t}\ \frac{dh^{(0)}(\beta,u)}{h^{(1)}(\beta,u)}\biggr\} \\
&& \qquad - \ \exp\biggl\{-\int_{-\infty}^{t}\ \frac{dH_n^{(0)}(\beta,u)}{H_n^{(1)}(\beta,u)}- n^{-2}\sum_{i:\epsilon_{\beta,i}\leq t}O\bigl(H_n^{(1)}(\beta,\epsilon_{\beta,i})^{-2}\bigr)\biggr\}\biggr|\\
&&\!\!\! \leq \biggl|\int_{-\infty}^{t}\ \frac{dH_n^{(0)}(\beta,u)}{H_n^{(1)}(\beta,u)}-\int_{-\infty}^{t}\ \frac{dh^{(0)}(\beta,u)}{h^{(1)}(\beta,u)}
+n^{-2}\sum_{i:\epsilon_{\beta,i}\leq t}O\bigl(H_n^{(1)}(\beta,\epsilon_{\beta,i})^{-2}\bigr)\biggr|.
\end{eqnarray*}
Under the condition $H_n^{(1)}(\beta,t)\geq n^{-\varepsilon}$, we have
$$ n^{-2}\sum_{i:\epsilon_{\beta,i}\leq t}O\bigl(H_n^{(1)}(\beta,\epsilon_{\beta,i})^{-2}\bigr)\leq n^{-2} \cdot O(n^{2\varepsilon})\cdot n = O(n^{-1+2\varepsilon})= o(n^{-\frac{1}{2}+3\varepsilon}).$$
So in order to show (\ref{eq:thm1.1}), we only need to show
\begin{eqnarray} \label{eq:thm1_eqn}
&& \sup\biggl\{\biggl|\int_{-\infty}^{t}\ \frac{dH_n^{(0)}(\beta,u)}{H_n^{(1)}(\beta,u)}-\int_{-\infty}^{t}\ \frac{dh^{(0)}(\beta,u)}{h^{(1)}(\beta,u)}\biggr|: \\
&&\hspace{1.5in} \beta \in \mathcal{B}, H_n^{(1)}(\beta,t)\geq n^{-\varepsilon}\biggr\}=o(n^{-\frac{1}{2}+3\varepsilon}) \nonumber
\end{eqnarray}
almost surely.
Now we define $ \tilde{T}_n = \sup\{t: \beta\in \mathcal{B}, H_n^{(1)}(\beta,t) \geq n^{-\varepsilon}\}$,
and let
\begin{equation*}
\tilde{H}_n^{(1)}(\beta,t) = \begin{cases}
                H_n^{(1)}(\beta,t), & \mbox{if $t \leq \tilde{T}_n$,} \\
                H_n^{(1)}(\beta,\tilde{T}_n), & \mbox{if $t > \tilde{T}_n$.}
            \end{cases}
\end{equation*}
Then $\tilde{H}_n^{(1)}(\beta,t)\geq n^{-\varepsilon}$ for all $\beta\in \mathcal{B}$ and $-\infty<t<\infty$. Define $\tilde{h}^{(1)}(\beta,t)$ similarly as $\tilde{H}_n^{(1)}(\beta,t)$ and apply Lemma \ref{lem:lem3} to $U_n(\beta,u) = n^{-2\varepsilon}\{\tilde{H}_n^{(1)}(\beta,u)\}^{-1}$ and $u_n(\beta,u) = n^{-2\varepsilon}\{\tilde{h}^{(1)}(\beta,u)\}^{-1}$, we obtain (\ref{eq:thm1_eqn}) and thus (\ref{eq:thm1.1}) holds.

We now show (\ref{eq:thm1.2}). Notice that $F(t)=F(\beta_0,t)$, then under the restriction $\{|\beta-\beta_0| \leq n^{-3\varepsilon},h^{(1)}(\beta,t)\geq n^{-\varepsilon}\}$, by the mean value theorem we obtain
\begin{eqnarray*}
&& \hspace{-0.2in} |F(\beta,t)-F(t)| \\
&& = \biggl|\exp\biggl\{-\int_{u\leq t}\ \frac{dh^{(0)}(\beta,u)}{h^{(1)}(\beta,u)}\biggr\} - \exp\biggl\{-\int_{u\leq t}\ \frac{dh^{(0)}(\beta_0,u)}{h^{(1)}(\beta_0,u)}\biggr\} \biggr| \\
&&  \leq \biggl|\int_{u\leq t}\ \frac{dh^{(0)}(\beta,u)}{h^{(1)}(\beta,u)} - \int_{u\leq t}\ \frac{dh^{(0)}(\beta_0,u)}{h^{(1)}(\beta_0,u)} \biggr| \\
&&  \leq \biggl|\int_{u\leq t}\ \frac{d\{h^{(0)}(\beta,u)-h^{(0)}(\beta_0,u)\}}{h^{(1)}(\beta,u)}\biggr| \\
&& \qquad +\ \biggl|\int_{u\leq t}\biggl(\frac{h^{(1)}(\beta_0,u)-h^{(1)}(\beta,u)}{h^{(1)}(\beta,u)h^{(1)}(\beta_0,u)}\biggr)\ dh^{(0)}(\beta_0,u)\biggr| \\
&&  \leq n^{\varepsilon}\sup\{|h^{(0)}(\beta_0,t)-h^{(0)}(\beta,t)|\} \\
 && \qquad + \ n^{2\varepsilon}h^{(0)}(\beta_0,t) \sup\{|h^{(1)}(\beta_0,t)-h^{(1)}(\beta,t)|\} \\
&&
= O(n^{-\varepsilon}),
\end{eqnarray*}
where the third inequality holds because for any $u\leq t$, $\{h^{(1)}(\beta,u)\}^{-1} \leq \{h^{(1)}(\beta,t)\}^{-1} \leq n^{\varepsilon}$, and the last equality holds because $h^{(0)}(\beta_0,t)\leq 1$ and $\sup\{|h^{(k)}(\beta,t)-h^{(k)}(\beta_0,t)|\} = O(|\beta-\beta_0|)$, $k=0,1$, by Lemma \ref{lem:lem2}. Thus (\ref{eq:thm1.2}) holds. Finally, (\ref{eq:thm1.3}) can be easily obtained by applying the triangle inequality to (\ref{eq:thm1.1}) and (\ref{eq:thm1.2}) together with Lemma \ref{lem:lem1} provided that $ -\frac{1}{2}+3\varepsilon \leq -\varepsilon $, i.e., $0<\varepsilon \leq \frac{1}{8}$.
\hfill $\Box$

\subsection{Proof of Theorem \ref{thm:consis2}}
Notice that
\begin{eqnarray*}
\alpha_0 = \int_{-\infty}^{\infty}t\, d F(t) = \int_{0}^{\infty}\{1-F(t)\}\ dt - \int_{-\infty}^{0}F(t)\, dt.
\end{eqnarray*}
We thus have
\begin{eqnarray} \label{eq:diff_alpha}
\int_{-\infty}^{\infty}t\, d\hat{F}_{n,\beta}(t)-\alpha_0 &=& \int_{-\infty}^{\infty}t\, d\hat{F}_{n,\beta}(t)-\int_{-\infty}^{\infty}t\, d F(t) \nonumber \\
&=& \biggl\{\int_{0}^{\infty}\{1-\hat{F}_{n,\beta}(t)\}\, dt - \int_{0}^{\infty}\{1-F(t)\}\, dt \biggr\} \\
&& \qquad - \ \biggl\{\int_{-\infty}^{0}\hat{F}_{n,\beta}(t)\, dt - \int_{-\infty}^{0}F(t)\, dt \biggr\}. \nonumber
\end{eqnarray}
With $\beta_0 \ne 0$, when $\beta$ satisfies $|\beta-\beta_0|\leq n^{-3\varepsilon}$, we have $\beta \ne 0$ for sufficiently large $n$. For any $\beta \ne 0$ and $t \in (-\infty, \infty)$, one can always find a range of $x$ such that $\bar{G}_{C|X}(t+\beta x|X=x)>0$ and $\bar{F}(t+(\beta-\beta_0)x)>0$ since $\bar{F}(t)>0$ for all $t<\infty$ under the assumption $f_X(x)>0$ for all $-\infty<x<\infty$ and $\beta_0 \ne 0$.
Therefore, we have $ h^{(1)}(\beta,t)> 0$ for all $t \in (-\infty, \infty)$ from the following equation that is obtained in the proof of Lemma \ref{lem:lem2}:
\begin{equation*}
h^{(1)}(\beta,t) = \int_{-\infty}^{\infty} \bar{F}(t+(\beta-\beta_0)x)\bar{G}_{C|X}(t+\beta x|X=x)f_X(x)\, dx.
\end{equation*}
Moreover, since $H_n^{(1)}(\beta,t) \rightarrow h^{(1)}(\beta,t)$ almost surely as $n \rightarrow \infty$, then with $n$ sufficiently large, we have $H_n^{(1)}(\beta,t)>0$ almost surely for any $\beta \ne 0$ and $t \in (-\infty, \infty)$. Hence $T_n \rightarrow \infty$ almost surely as $n \rightarrow \infty$, where $T_n = \sup\big\{t: H_n^{(1)}(\beta,t) \geq n^{-\varepsilon}, |\beta-\beta_0|\leq n^{-3\varepsilon} \big\}$, as defined in (\ref{eq:Tn}).

Then at $\beta=\beta_0$, by the independence of $e_0$ and $C-\beta_0 X$ and the Markov's inequality, it follows that
\begin{eqnarray*}
h^{(1)}(\beta_0,T_n) &=& P\{1(e_0 \geq T_n)\}\cdot P\{1(C-\beta_0 X \geq T_n)\} \\
&\leq& P\{1(e_0 \geq T_n)\} \leq \frac{Ee_0^2}{T_n^2}.
\end{eqnarray*}
Since $H_n^{(1)}(\beta_0,T_n) \geq n^{-\varepsilon}$ implies $h^{(1)}(\beta_0, T_n) \geq n^{-\varepsilon}$, together with Condition 4 that $Ee_0^2<\infty$, we have $T_n^2 \leq Ee_0^2\{h^{(1)}(\beta_0,T_n)\}^{-1} \leq O(n^{\varepsilon})$, i.e., $T_n \leq O(n^{\varepsilon/2})$. This implies that $T_n \rightarrow \infty$ in a rate no faster than $n^{\varepsilon/2}$.

Since the Kaplan-Meier estimator $\hat{F}_{n,\beta}(t)$ is set to 1 for $t>T_n$, equation (\ref{eq:diff_alpha}) becomes
\begin{eqnarray*}
\int_{-\infty}^{\infty}t\ d\hat{F}_{n,\beta}(t)-\alpha_0 &=&
\int_{0}^{T_n}\{F(t)-\hat{F}_{n,\beta}(t)\}\, dt - \int_{T_n}^{\infty}\{1-F(t)\}\, dt \\
&& \qquad - \ \int_{-\infty}^{0}\{\hat{F}_{n,\beta}(t)-F(t)\}\, dt.
\end{eqnarray*}
Then by Theorem \ref{thm:consis1}, we have
$$
\sup \biggl\{\int_{0}^{T_n}|F(t)-\hat{F}_{n,\beta}(t)|\, dt: |\beta-\beta_0|\leq n^{-3\varepsilon} \biggr\} \leq T_n \cdot O(n^{-\varepsilon}) \leq O(n^{-\frac{\varepsilon}{2}})
$$
almost surely. For the second term on the right hand side of above equation, applying the Markov's inequality we obtain
$$
\int_{T_n}^{\infty}\{1-F(t)\}\, dt \leq \int_{T_n}^{\infty} P\{1(|e_0|\geq t)\}\, dt \leq \int_{T_n}^{\infty} \frac{Ee_0^2}{t^2}\, dt \leq \frac{Ee_0^2}{T_n} = o(1)
$$
almost surely. For the third term, we have
\begin{eqnarray*}
\int_{-\infty}^{0}\{\hat{F}_{n,\beta}(t)-F(t)\}\, dt &=& \int_{-T_n}^{0}\{\hat{F}_{n,\beta}(t)-F(t)\}\, dt \\
&& \qquad + \int_{-\infty}^{-T_n}\{\hat{F}_{n,\beta}(t)-F(t)\}\, dt,
\end{eqnarray*}
where
$$
\sup \biggl\{\int_{-T_n}^{0}|F(t)-\hat{F}_{n,\beta}(t)|\, dt: |\beta-\beta_0|\leq n^{-3\varepsilon} \biggr\} \leq T_n \cdot O(n^{-\varepsilon}) \leq O(n^{-\frac{\varepsilon}{2}})
$$
almost surely, and
\begin{eqnarray*}
\int_{-\infty}^{-T_n}|F(t)-\hat{F}_{n,\beta}(t)|\, dt &\leq& \int_{-\infty}^{-T_n} F(t)\, dt + \int_{-\infty}^{-T_n} \hat{F}_{n,\beta}(t)\, dt \\
& =& \int_{T_n}^{\infty}F(-t)\, dt + \int_{-\infty}^{-T_n} \hat{F}_{n,\beta}(t)\, dt \\
&\leq & \frac{Ee_0^2}{T_n} + o(1) = o(1)
\end{eqnarray*}
almost surely, where the last inequality holds because of the Markov's inequality
$$F(-t)=P\{1(e_0 \leq -t)\}\leq P\{1(|e_0|\geq t)\} \leq \frac{Ee_0^2}{t^2}$$
and the fact $\int_{-\infty}^{-T_n} \hat{F}_{n,\beta}(t)\ dt \rightarrow 0$ almost surely as $n \rightarrow \infty$. Therefore,
\begin{equation*}
\sup\biggl\{\biggl|\int_{-\infty}^{\infty}t \, d\hat{F}_{n,\beta}(t)-\alpha_0 \biggr|: |\beta-\beta_0| \leq n^{-3\varepsilon} \biggr\}=o(1)
\end{equation*}
almost surely. We now have proved Theorem \ref{thm:consis2}.
\hfill $\Box$

\bibliographystyle{sjs}
\bibliography{biblio_YD}



\end{document}